\pdfoutput=1
\documentclass[11pt]{amsart}
\usepackage{amsmath,amssymb,amsthm,mathtools}
\usepackage{enumitem}
\usepackage{needspace}
\usepackage[hidelinks]{hyperref}
\newtheorem{theorem}{Theorem}[section]
\newtheorem{proposition}[theorem]{Proposition}
\newtheorem{lemma}[theorem]{Lemma}
\newtheorem{corollary}[theorem]{Corollary}
\theoremstyle{definition}

\newtheorem{example}[theorem]{Example}
\newtheorem{remark}[theorem]{Remark}
\newcommand{\Z}{\mathbb Z}
\newcommand{\F}{\mathbb F}
\newcommand{\C}{\mathbb C}
\newcommand{\cH}{\mathcal H}
\newcommand{\Mod}{\operatorname{Mod}}
\newcommand{\rot}{\operatorname{rot}}
\newcommand{\Hom}{\operatorname{Hom}}
\newcommand{\id}{\operatorname{id}}
\newcommand{\one}{\mathbf 1}
\title[Non-planar Calabi--Yau Fillings and Quantum Computation]{Non-planar
Contact Structures with Calabi--Yau Fillings and Topological Quantum
Computation}
\author{Atsuhide Mori}
\address{Department of Mathematics, Osaka Dental University}
\email{mori-a@cc.osaka-dent.ac.jp}
\subjclass[2020]{Primary 57K33; Secondary 57M12, 53D35, 20F36, 81P68}
\keywords{non-planar open book, connected binding, Stein filling,
positive factorization, Calabi--Yau filling, simple branched cover,
liftable braid, Ising representation, Fibonacci representation}

\begin{document}
\begin{abstract}
This paper is a non-planar counterpart of the author's earlier work on
planar contact structures and topological quantum computation.  On a
planar page every branch point is paired by a liftable arc, and the
pairing matches the Ising representation exactly, in that fixing the
channels of the pairs already fixes the state.  Here the page has genus $g$ and
connected binding, presented as a simple degree-three cover of the
disk.  Every maximal system of disjoint liftable arcs then has exactly
$g$ members and leaves two branch points unused, and the half-twist
interchanging the two points is not liftable, although its cube is and
lifts to the identity.  Ising fusion does not see that leftover pair,
determining its channel from the other $g$ and acting there only by
scalars, whereas Fibonacci fusion can leave the channel undetermined.
The powers of the cube then act by ten distinct relative phases on a
single joint eigenspace, while the Calabi--Yau Stein filling stays
fixed.  An integral residue criterion supplies the Calabi--Yau test, a
factorization localized on a contact connected sum is automatically
Calabi--Yau, and, unlike in the planar case, a Calabi--Yau
factorization need not be the shortest.
\end{abstract}
\maketitle

\section{Introduction}

This paper is a non-planar counterpart of the author's earlier work on
planar contact structures \cite{PlanarCY}.  On a planar page every
branch point is paired by a liftable arc, and that pairing matches the
Ising representation exactly.  Here the page has positive genus and the
binding is connected, so that two branch points are left over.  The
leftover pair is invisible to Ising fusion and visible to Fibonacci
fusion, and that is where the two models part.  Along the way we
examine the Calabi--Yau condition on the Stein fillings that the
geometry supplies.

The quantum side of the question concerns anyons, particles in two
dimensions whose exchanges are recorded by braids rather than by
permutations.  A configuration of several anyons is
described by the channels into which they can fuse, and a braid acts
unitarily on the space of these channels \cite{NSSFD}.  Two models
dominate the subject, Ising and Fibonacci.  Ising anyons are the ones
associated with Majorana modes, and braiding them alone is not
universal for quantum computation, whereas Fibonacci braiding is
\cite{FLW}.  Both appear below, on the same cover and the same system of arcs.

A braid representation assigns quantum operations to braids.
Lifting a suitable braid through a branched cover gives the monodromy of
an open book \cite{Giroux}, and a positive allowable factorization of
that monodromy specifies a Stein filling.
We take the topology relevant to a quantum operation to be the lift and
its factorization, and not the braid alone.  This datum is finer than
the braid, and the geometry accepts it only for positive factorizations,
so the invertibility of the braid group is given up.  It is not needed.
The representations used below have finite order on each generator, so
an inverse can be traded for a positive power, and the monoid of
positive liftable words already realizes every operation, exactly for
Fibonacci and projectively for Ising (Section~\ref{sec:realization}).
Following \cite{PlanarCY}, we call a Stein filling $W$ Calabi--Yau (CY)
when
\[
                 c_1(W)=0\quad\text{in }H^2(W;\Z),
\]
and we call a positive allowable factorization CY when the Stein filling
it determines is CY.

The planar case shows why Ising is the model there.  A planar page with
$d$ boundary components, presented as a simple degree-$d$ cover of the
disk, has $2d-2$ branch points, and all of them are paired by $d-1$
disjoint liftable arcs whose lifted curves form a basis of the first
homology of the page.  On the quantum side the pair channels satisfy a
single total-parity relation, and the resulting joint eigenlines
exhaust the Ising state space, so the geometric pairing is already a
complete measurement.  Admissible cuts group these lines into
sectorwise tensor products \cite{PlanarCY}, a direction pursued in
\cite{MeasurementCuts}.

Concretely, fix a simple cover of degree three,
\[
 p:P=\Sigma_{g,1}\longrightarrow D^2,\qquad g\geq1.
\]
Its boundary monodromy is a $3$-cycle, and the number of branch points is
$n=2g+2$.
Here non-planar refers to the chosen page. Its genus is not assumed to be
minimal among the open books supporting the contact structure.
The first difference from the planar construction is that a maximal
system of disjoint liftable arcs pairs only $2g$ of the points.
Its lifted curves span a Lagrangian subspace of $H_1(P;\F_2)$ rather than
the whole of it, and two branch points remain.
A half-twist interchanging these two points, chosen disjoint from the
maximal system, is not liftable. Its cube is liftable and lifts to the
identity.

The question is whether this remaining pair can carry quantum data once
the liftable pair channels have been fixed, and the answer separates the
two models.  Ising fusion determines the last pair channel from the
other $g$ channels and the total charge, whereas Fibonacci fusion need
not.
In the Fibonacci case the cube at the remaining pair acts by a relative
phase within a fixed joint channel space.
Section~\ref{sec:comparison} proves this, and
Section~\ref{sec:realization} realizes the action while keeping one CY
Stein filling fixed.

Two geometric results are needed for this.  Section~\ref{sec:arcs}
proves the maximality statement and identifies the local kernel of the
lifting map, and Section~\ref{sec:cy} gives an integral CY criterion
adapted to the cover, with its consequences for cutting and for
minimality.

\section{Liftable pairs and the two remaining points}\label{sec:arcs}

Choose sheet labels $0,1,2$ and a Hurwitz system for the branch set $\Delta$.
The cover is described by a tuple of transpositions
\[
                 (t_1,\ldots,t_n),\qquad t_1\cdots t_n=C,
\]
where $C$ is a fixed $3$-cycle.  The Riemann--Hurwitz formula gives
$\chi(P)=3-n=1-2g$.  Write $L_p\subset B_n$ for the stabilizer of this
tuple under the Hurwitz action and
\[
                 \ell_p:L_p\longrightarrow\Mod(P\operatorname{rel}\partial P)
\]
for the lifting homomorphism.  The boundary is fixed pointwise.  For an
embedded arc $a$ joining two branch points and otherwise avoiding
$\Delta$, $h_a$ denotes its positive half-twist,
and $\tau_\gamma$ denotes the right-handed Dehn twist about $\gamma$.
The local monodromies at the ends of $a$ are transported along the arc,
and they depend on the arc, not just on its endpoints.  A product of positive
half-twists is called \emph{quasipositive}. The arcs need not be standard
adjacent arcs in a fixed braid coordinate system.  As in \cite{PlanarCY},
such a factorization is \emph{factorwise $p$-liftable} when every
half-twist factor belongs to $L_p$.

The following local alternative is standard in the theory of liftable
braids \cite[Lemma~2.3]{MulazzaniPiergallini}. We include the argument with the
boundary convention used here.

\begin{lemma}[The local alternative]\label{lem:local}
Let $D_a$ be a closed regular neighborhood of $a$ containing no other
branch point.
\begin{enumerate}[label=\textup{(\roman*)}]
\item If the two transported transpositions agree, $p^{-1}(D_a)$ is an
annulus together with a disk.  The half-twist is liftable and lifts to
the positive Dehn twist about the annular core $\gamma_a$.
\item If they differ, $p^{-1}(D_a)$ is a disk.  The local liftable group
is $\langle h_a^3\rangle$, and $\ell_p(h_a^3)=1$.
\end{enumerate}
\end{lemma}

\begin{proof}
Equal transpositions generate a group with an orbit of size two and a
fixed sheet, so $p^{-1}(D_a)$ is an annulus together with a disk, and the
annular lift of the half-twist is a positive Dehn twist.  Distinct
transpositions generate $S_3$, so the preimage is connected, of Euler
characteristic one, and has one boundary circle because the product is a
$3$-cycle.  It is a disk.  The Hurwitz move
$(t,t')\mapsto(tt't^{-1},t)$ cycles through the three ordered pairs with
that product, so exactly the powers divisible by three are liftable, and
their lifts fix the boundary of a disk.
\end{proof}

We call an arc in case (i) \emph{liftable}.  A \emph{pair system} is a
collection of pairwise disjoint liftable arcs, with disjoint endpoints.
Maximality is with respect to inclusion.

\begin{theorem}[Maximal pair systems]\label{thm:pairs}
Every pair system extends to one of $g$ arcs, and none has more than $g$
arcs.  The lifted curves of every pair system are independent over
$\F_2$, and for a maximal system they span a Lagrangian of $H_1(P;\F_2)$.
Exactly two branch points remain.  Any arc
joining them and disjoint from the system is in case (ii) of
Lemma~\ref{lem:local}.
\end{theorem}

\begin{proof}
Remove disjoint regular neighborhoods of the arcs of a pair system
$a_1,\ldots,a_k$ from the base disk.  The monodromy group of the
remaining region contains the boundary $3$-cycle $C$, so its preimage is
connected, and $P$ cut along the $k$ annular cores is that preimage with
two half-annuli and a disk attached for each removed neighborhood, each
along a single boundary circle.  The cut surface is connected, so the
core classes are independent over $\F_2$.

There cannot be $g+1$ arcs.  Such a system would use all $2g+2$ branch
points, and in an adapted Hurwitz system its tuple would consist of
equal adjacent pairs, with product the identity rather than $C$.

For completion we use an elementary normal form.  Every tuple of $2h+2$
transpositions with product $C$ is Hurwitz equivalent to
\begin{equation}\label{eq:normal}
              (x_1,x_1,\ldots,x_h,x_h,A,B),\qquad A\ne B.
\end{equation}
Among the first three entries of a tuple of length at least four a
Hurwitz move produces an equal adjacent pair, which travels to the front
by the two-step moves $(a,x,x)\sim(axa,axa,a)$.  Removing it leaves a
shorter tuple with the same product, and length two is the base case.

Now put the given $k<g$ arcs into standard position, by the change of
coordinates principle.  The corresponding entries of the tuple are then
$k$ equal adjacent pairs, and the remaining $2(g-k)+2$ entries still
have product $C$.  Applying the normal form in a disk disjoint from the
given arcs and pulling back the $g-k$ additional standard pairs gives
the extension.

The $g$ lifted curves are disjoint and independent, so their classes
span an isotropic subspace of half the dimension of $H_1(P;\F_2)$, and
the intersection form is nondegenerate because $P$ has one boundary
component.  Finally an arc joining the unused points exists in the
complement of the system, and if it were liftable it could be added,
contradicting the upper bound.
\end{proof}

Fix a maximal system $\mathcal M=\{a_1,\ldots,a_g\}$ and choose such an
unused-pair arc $a_0$, with closed regular neighborhood disjoint from
$\mathcal M$.  Put $h_i=h_{a_i}$, including $h_0=h_{a_0}$.
Lemma~\ref{lem:local} gives
\begin{equation}\label{eq:kernel}
 L_p\cap\langle h_0\rangle=\langle h_0^3\rangle
                \subset\ker\ell_p.
\end{equation}
The statement is local and requires the arc $a_0$ to avoid $\mathcal M$.
Its endpoints alone do not determine its type.
The union of $a_0$ and $\mathcal M$ is a complete matching of the branch points,
but only the $g$ arcs in $\mathcal M$ lift to Dehn twists.

\begin{remark}\label{rem:notinH}
For $g\le2$, the element $h_0^3$ lies outside the subgroup $H$ generated
by all liftable half-twists.  To see this, let $A$ be the abelianization
of $\Mod(P\operatorname{rel}\partial P)$ and let $x\in A$ be the common
class of the twists about nonseparating curves.  Each generator of $H$
has writhe one and lifts to such a twist, so the class of
$\ell_p(\alpha)$ in $A$ is $(\operatorname{writhe}\alpha)x$ for every
$\alpha\in H$.  Here $A\cong\Z$ for $g=1$, since
$\Mod(\Sigma_{1,1}\operatorname{rel}\partial)\cong B_3$
\cite{FarbMargalit}, and $A\cong\Z/10$ for $g=2$ \cite{Korkmaz}.
In both cases $x$ generates $A$ and $3x\ne0$, whereas $h_0^3$ has writhe
three and lifts to the identity.  For $g\ge3$ the abelianization is
trivial, so this argument does not decide whether $h_0^3\in H$.
\end{remark}

\section{Stein fillings and the Calabi--Yau condition}\label{sec:cy}

A positive allowable factorization
$\lambda=(\tau_{\gamma_1},\ldots,\tau_{\gamma_M})$ of a monodromy $\phi$
on $P$ determines a positive allowable Lefschetz fibration (PALF), and
hence a Stein filling
$W_\lambda=W_{\gamma_1,\ldots,\gamma_M}$ \cite{LoiPiergallini}.
Here allowable means that each $[\gamma_j]$ is nonzero in $H_1(P;\Z)$.
Fix a global framing $v$ of $TP$.  The handle description identifies
$H^2(W_\lambda;\Z)$ with the cokernel of $u\mapsto(u([\gamma_j]))_j$ on
$H^1(P;\Z)$, and $c_1(W_\lambda)$ with the class of
$(\rot_v\gamma_j)_j$ \cite[Proposition~2.3]{Gompf}, so
\begin{equation}\label{eq:gompf}
 W_\lambda\text{ is CY}\quad\Longleftrightarrow\quad
 \exists u\in H^1(P;\Z):\ u([\gamma_j])=\rot_v\gamma_j\quad\forall j.
\end{equation}
In the planar case of \cite{PlanarCY} every rotation number is one.
The criterion is integral, and replacing $\Z$ by $\mathbb Q$ discards
torsion obstructions.

\begin{remark}[Every allowable factorization is a lift]\label{rem:braidderived}
The lifting homomorphism for a simple three-fold cover of the disk is
surjective \cite{MontesinosMorton}. See also the introduction of
\cite{Zuddas}.  Let $c$ be any nonseparating simple
closed curve in $P$, and fix a liftable arc $a$.  The mapping class
group acts transitively on nonseparating curves, so choose
$f\in\Mod(P\operatorname{rel}\partial P)$ with $f(\gamma_a)=c$, and then
$b\in L_p$ with $\ell_p(b)=f$.  The arc $b(a)$ is liftable and its lifted
curve is isotopic to $c$.  Since $P$ has one boundary component, every
allowable vanishing cycle is nonseparating.  Applying this construction
to each cycle realizes any positive allowable factorization on $P$ as
the factorwise lift of a quasipositive braid, preserving the ordered
list of vanishing cycles.  The criterion below therefore applies to all
such factorizations, including both sides of
Theorem~\ref{thm:notminimal}.
\end{remark}

\subsection*{A framing supplied by the cover}
Let $S=p^{-1}(\Delta)$, let $\mathcal R\subset S$ be the ramification
points, and put $U=P\setminus S$.  Lift a constant vector field on the disk horizontally
to a nonvanishing field $X$ on $U$.  Near a ramification point, in a
coordinate with $p(w)=w^2$, it is $(2w)^{-1}\partial_w$ and has index
$-1$.  It extends without zeros across the unramified points of $S$.

For a liftable arc $a$, take the boundary of its small neighborhood in
the base.  Either annular boundary component over it is isotopic to
$\gamma_a$ in $P$ and maps diffeomorphically to that boundary.  Choose
one and orient its projection positively.  On this representative,
\begin{equation}\label{eq:rotation}
                         \rot_X\gamma_a=1,
\end{equation}
because the tangent turns once relative to the constant field
downstairs.  The singularities of $X$ are the price of this planar
normalization.

\begin{theorem}[Integral residue criterion]\label{thm:cy}
Suppose the vanishing cycles are lifts of liftable arcs, represented as
in \eqref{eq:rotation}.  Then $W_\lambda$ is CY if and only if there is
$w\in H^1(U;\Z)$ such that
\begin{equation}\label{eq:residues}
 w(\mu_q)=
 \begin{cases}1&q\in\mathcal R,\\0&q\in S\setminus\mathcal R,\end{cases}
       \qquad w([\gamma_j])=1\quad(1\le j\le M),
\end{equation}
where $\mu_q$ is a positively oriented meridian of $q$.
\end{theorem}

\begin{proof}
Choose a global framing $v$ and write $\xi=[X/v]\in H^1(U;\Z)$, so that
$\rot_X\gamma=\rot_v\gamma-\xi([\gamma])$ for every loop in $U$.  The
local indices give $\xi(\mu_q)=-1$ at ramification points and zero at
the others.  If $u$ solves \eqref{eq:gompf}, then $w=i^*u-\xi$, for
$i:U\hookrightarrow P$, has the required residues and evaluates to one
on each vanishing cycle by \eqref{eq:rotation}.

Conversely, evaluation on meridians gives the exact sequence
\[
0\longrightarrow H^1(P;\Z)\xrightarrow{i^*}H^1(U;\Z)
 \longrightarrow\Z^S\longrightarrow H^2(P;\Z)=0,
\]
so a $w$ satisfying \eqref{eq:residues} has $w+\xi=i^*u$ with
$u([\gamma_j])=1+\xi([\gamma_j])=\rot_v\gamma_j$ for every $j$.
\end{proof}

The criterion becomes a finite integer system.  The unbranched cover
retracts to its Schreier graph, whose directed edges $e_{a,i}$ lift the
$i$-th meridian from sheet $a$.  Represent $w$ by an integral edge
cochain.  Order the two interchanged sheet labels as $a_i<b_i$ and let
$c_i$ be the fixed sheet label.  The residue equations take the form
\begin{equation}\label{eq:coordinates}
 w(e_{c_i,i})=0,\qquad
 w(e_{a_i,i})=s_i,\qquad w(e_{b_i,i})=1-s_i.
\end{equation}
Evaluating on the lifted arc chains gives a finite system of linear
equations in the $s_i$ with integer coefficients, and integer
solvability is decided by its Smith normal form.  For a standard pair with
equal transpositions the equation is simply $s_i=s_{i+1}$.

\begin{corollary}[No isotropic triple]\label{cor:isotropic}
Suppose $W_\lambda$ is CY and three vanishing cycles of $\lambda$ have
classes $x$, $y$ and $x+y$ in $H_1(P;\F_2)$.  Then $x\cdot y=1$.
In particular, the sum of the classes of two disjoint vanishing cycles
cannot be the class of a third vanishing cycle in the same factorization.
\end{corollary}

\begin{proof}
The spin structure induced by a global framing $v$ extends over the disk
capping the unique boundary of $P$, and under the resulting isomorphism
$H_1(P;\F_2)\cong H_1(\Sigma_g;\F_2)$ Johnson's quadratic refinement
reads $q_v([c])=1+\rot_v(c)$ modulo two for embedded simple closed
curves $c$ \cite{Johnson}.  If $u$ solves \eqref{eq:gompf}, a framing
$v'$ with $[v'/v]=u$ has
$\rot_{v'}(\gamma_j)=\rot_v(\gamma_j)-u([\gamma_j])=0$, so
$q_{v'}([\gamma_j])=1$ for every vanishing cycle.  The identity
$q_{v'}(x+y)=q_{v'}(x)+q_{v'}(y)+x\cdot y$ now gives $1=1+1+x\cdot y$.
\end{proof}

A factorization with a single nonseparating vanishing cycle is always
CY, because its primitive homology class permits any prescribed integral
evaluation in \eqref{eq:gompf}.

\subsection*{What the condition contributes to cutting}
Let $\beta\in L_p$ have a specified factorization into positive liftable
half-twists.  A proper arc $\delta$ in the base, avoiding $\Delta$, will be
called an \emph{admissible cut for this factorization} when the two
blocks, the components of $D^2\setminus\delta$, both contain branch
points, $\beta$ is represented by the identity near
$\delta$, each half-twist is supported in one block, and the lifted
gluing graph is a tree.  The graph has one vertex per component of the
preimages of the blocks and one edge per component of $p^{-1}(\delta)$.
These assumptions specify a cut of the open book and its given filling,
rather than just a cut of the page.

\begin{proposition}[A Calabi--Yau obstruction to cutting]\label{prop:cut}
An admissible cut has two blocks of odd size, each \emph{monochromatic} in
the sense that all its transpositions agree.  It realizes
the page as a boundary connected sum, and the contact manifold as the
contact connected sum of the pieces.  Its localized factorization is CY.
Consequently a factorization which is not CY admits no such cut.
\end{proposition}

\begin{proof}
A nonempty block has two preimage components exactly when all its
transpositions agree, the preimage being a connected double cover
together with a disk, and one otherwise.  The lifted cut has three
components, so the connected gluing graph has three edges and is a tree
exactly when the two component counts sum to four.  Both blocks are
therefore monochromatic, with products $x^r$ and $y^s$ satisfying
$x^ry^s=C$, which forces $r,s$ odd and $x\ne y$.  The pieces glue along
the tree by boundary connected sums of pages, and each half-twist is
supported in one block, so the monodromy is a product of factors
supported in the two sides and is the identity near the gluing arcs.
This gives the contact connected sum.

On a monochromatic block label the interchanged sheets $0,1$ and the
fixed sheet $2$, and set $w(e_{0,i})=0$, $w(e_{1,i})=1$ and
$w(e_{2,i})=0$.  The residues are as required.  A lifted arc chain is a
cycle on the two interchanged sheets whose projected word has total
exponent two, so it is traversed once from each of them and
$w(\gamma)=1$.  Theorem~\ref{thm:cy} now applies componentwise to the
restricted cover, and the tree splits $H^1(P;\Z)$ as the direct sum of
the cohomologies of the pieces, so the solutions combine.
\end{proof}

Degree three also limits the number of simultaneous cuts.

\begin{proposition}[One cut at a time]\label{prop:onecut}
Let $\delta_1,\ldots,\delta_m$ be disjoint proper arcs in the base
avoiding $\Delta$, dividing $D^2$ into $m+1$ blocks each containing a branch
point, and suppose the lifted gluing graph is a tree.  Then $m\le1$.
\end{proposition}

\begin{proof}
The preimage of each $\delta_i$ has three components, so the graph has
$3m$ edges, and it has $\sum_ic_i$ vertices, where $c_i$ is the number of
preimage components of the $i$-th block.  A tree satisfies
$\sum_ic_i=3m+1$.  As in the proof of Proposition~\ref{prop:cut} a
nonempty block has $c_i\le2$, so $3m+1\le2(m+1)$ and $m\le1$.
\end{proof}

In the planar construction, several simultaneous cuts give sectorwise
tensor products \cite{PlanarCY} (see also \cite{MeasurementCuts}).

The CY condition is not automatic on a page of positive genus.
In $\Sigma_{2,1}$ embed
$Q\cong\Sigma_{1,2}$ whose complement is a pair of pants, and orient its
boundary curves $d,e$, both nonseparating in $\Sigma_{2,1}$, so that
$[d]=[e]$.  The boundary rotation formula
gives $\rot_v(d)-\rot_v(e)=\chi(Q)=-2$.  Thus the positive allowable
factorization $\tau_d\tau_e$ fails \eqref{eq:gompf} even over $\mathbb Q$.
The obstruction concerns the localization of a factorization rather
than a classification of all surface decompositions.

\subsection*{Calabi--Yau no longer means shortest}
The same subsurface says more.  In the planar construction of
\cite{PlanarCY} a Calabi--Yau positive factorization is shortest among
the positive factorizations of its monodromy.  On a page of positive
genus, the Calabi--Yau condition no longer has this consequence.

\begin{theorem}[CY factorizations need not be shortest]\label{thm:notminimal}
Let $a,b,c$ be a chain of three curves in the subsurface
$Q\cong\Sigma_{1,2}$ above.  The chain relation
\cite[Proposition~4.12]{FarbMargalit}
\[
                 (\tau_a\tau_b\tau_c)^4=\tau_d\tau_e
\]
holds in $\Mod(\Sigma_{2,1})$.  Both sides are positive allowable
factorizations of the same monodromy.  The left-hand side, of length
twelve, is Calabi--Yau; the right-hand side, of length two, is not.
\end{theorem}

\begin{proof}
All five curves are nonseparating in $\Sigma_{2,1}$, so both
factorizations are allowable.  The regular neighborhood $Q$ of
$a\cup b\cup c$ retracts onto that graph, so $[a],[b],[c]$ form a basis
of $H_1(Q;\Z)$.

Write $\Sigma_{2,1}=Q\cup_{d\sqcup e}P'$ with $P'$ the complementary pair
of pants.  Since $H_1(P';\Z)$ is generated by $[d]$ and $[e]$, which are
classes of curves lying in $Q$ as well, Mayer--Vietoris makes
$H_1(Q;\Z)\to H_1(\Sigma_{2,1};\Z)$ injective with cokernel
$\widetilde H_0(d\sqcup e)\cong\Z$.  That cokernel is torsion-free, so
$H_1(Q;\Z)$ is a direct summand and the basis $[a],[b],[c]$ extends to
one of $H_1(\Sigma_{2,1};\Z)$.

Hence for a global framing $v$ there is $u\in H^1(\Sigma_{2,1};\Z)$ with
$u([a])=\rot_v a$, $u([b])=\rot_v b$ and $u([c])=\rot_v c$.  The
vanishing cycles of $(\tau_a\tau_b\tau_c)^4$ are $a,b,c$, each repeated
four times, so this single $u$ satisfies all twelve conditions of
\eqref{eq:gompf}.  That $\tau_d\tau_e$ fails them is the computation
above.
\end{proof}

The two fillings have different topology.  For the long factorization
$M=12$ and the classes of the vanishing cycles have rank three, so
$b_1=1$ and $b_2=9$. For the short one $M=2$ and the rank is one, so
$b_1=3$ and $b_2=1$.  One contact manifold therefore carries both a
Calabi--Yau and a non-Calabi--Yau Stein filling, with the longer
factorization giving the Calabi--Yau filling.

\begin{remark}[The planar lantern]\label{rem:lantern}
The planar lantern relation on $\Sigma_{0,4}$ equates positive allowable
factorizations of lengths four and three \cite{FarbMargalit}, and
neither side is CY.  To check this, orient every lantern curve
counterclockwise and use the planar framing and winding-number basis
$e_j=-[\partial_j\Sigma_{0,4}]$ of \cite{PlanarCY}, where
$\partial_1,\partial_2,\partial_3$ are the inner boundaries with their
boundary orientations, so that the four boundary-parallel classes are
$e_1,e_2,e_3,e_1+e_2+e_3$.  Writing $u_i=u(e_i)$, the three interior
curves have classes $e_1+e_2$, $e_1+e_3$, $e_2+e_3$ and rotation number
one.  Equation~\eqref{eq:gompf} requires $u_i+u_j=1$ for $i<j$,
whose only rational solution is $u_1=u_2=u_3=\tfrac12$.  The resulting
$c_1$ is a nonzero class of order two.  For the four boundary-parallel
curves all rotation numbers are one, so the conditions are
$u_1=u_2=u_3=1$ together with $u_1+u_2+u_3=1$.

The interior side illustrates why the CY test must be integral.
Theorem~\ref{thm:notminimal} exhibits the additional positive-genus
phenomenon that a CY factorization can be longer than a non-CY factorization
of the same monodromy.
\end{remark}

\section{The channel left by a maximal pair system}\label{sec:comparison}

We now keep the cover, $\mathcal M$ and $a_0$ fixed, and ask what the two
standard unitary braid representations on these $2g+2$ points, imported
rather than read off from the cover, see of the remaining pair.  We use
their fusion spaces and the projections onto the fusion channels of the
pairs, attached to the \emph{arcs}.  A
change of coordinates taking the disjoint arcs $a_0,a_1,\ldots,a_g$ to standard pairs identifies these projections
with the usual pair projections by unitary braid conjugation.

For Ising, write $x$ for the non-Abelian label, to distinguish it from
a braid generator, with rules
\[
 x\otimes x=\one\oplus\psi,\qquad
 \psi\otimes\psi=\one,\qquad \psi\otimes x=x.
\]
The state space is $\cH_n^{\mathrm I}(c)=\Hom(c,x^{\otimes n})$, with total charge
$c\in\{\one,\psi\}$ \cite{NSSFD}.  For $n=2g+2$ these are the two
parity sectors of \cite{PlanarCY}, labelled here by the total charge.
For Fibonacci the rules and spaces are
\[
 \tau\otimes\tau=\one\oplus\tau,\qquad
 \cH_n^{\mathrm F}(c)=\Hom(c,\tau^{\otimes n}),\qquad c\in\{\one,\tau\}.
\]
We set $\cH_0^{\mathrm F}(\one)=\C$ and $\cH_0^{\mathrm F}(\tau)=0$. For $n\ge1$,
\begin{equation}\label{eq:fibdim}
            \dim \cH_n^{\mathrm F}(\one)=F_{n-1},\qquad \dim \cH_n^{\mathrm F}(\tau)=F_n,
\end{equation}
where $F_0=0,F_1=1$.  These are the usual fusion rules and braid
representations \cite{TTWL}. The total charge is fixed independently of
the covering geometry.  We write $R$ for the unitary braid representation
in the chosen model and total-charge sector, and $\rho=[R]$ for its
projectivization.  In Ising, $\rho$ is the projective representation used
in \cite{PlanarCY}.  The chosen unitary normalization records the scalar
phases in the formulas below.  Fibonacci braiding is
universal for quantum computation \cite{FLW}, whereas Ising braiding
alone is not \cite{NSSFD}.

The projections for disjoint pair arcs commute, and we denote by $E$ a
joint eigenspace after fixing the channels of the $g$ \emph{liftable}
pairs in $\mathcal M$, the channel of $a_0$ remaining unfixed.  This
distinction between $g$ geometric pairs and all $g+1$ pairs is
essential.

\begin{theorem}[Ising and Fibonacci at the remaining pair]\label{thm:comparison}
For the cover and maximal pair system above:
\begin{enumerate}[label=\textup{(\roman*)}]
\item In Ising, each choice of the $g$ channels gives $\dim E=1$.
Those channels and the total charge determine the remaining channel.
\item In Fibonacci, suppose $m$ of the $g$ fixed channels are $\tau$.
There is a decomposition by the channel of $a_0$,
\begin{equation}\label{eq:E}
 E=E_{\one}\oplus E_\tau
       \cong \cH_m^{\mathrm F}(c)\oplus \cH_{m+1}^{\mathrm F}(c),
\end{equation}
and hence $E\cong \cH_{m+2}^{\mathrm F}(c)$.  Both summands are nonzero precisely when
$m\ge2$ for $c=\one$, and when $m\ge1$ for $c=\tau$.
\item Choose the Fibonacci chirality and scalar normalization with
$r_{\one}=e^{-4\pi i/5}$ and $r_\tau=e^{3\pi i/5}$, the complex
conjugates of \cite[Eq.~(2.7)]{TTWL}.  For $t\in\Z$, the local braids
$h_0^{3t}\in\ker\ell_p$ act on $E$ as
\begin{equation}\label{eq:relative}
 R(h_0^{3t})|_E
       =r_{\one}^{3t}\id_{E_{\one}}\oplus r_\tau^{3t}\id_{E_\tau}.
\end{equation}
When both summands are nonzero, $t=0,\ldots,9$ gives ten distinct
projective operators on this \emph{fixed} joint eigenspace.  In Ising
the same local braids act as scalars on each $E$.
\end{enumerate}
\end{theorem}

\begin{proof}
In Ising each pair channel is $\one$ or $\psi$, both invertible and with
single-valued fusion, so the total-charge condition
$q_0\otimes q_1\otimes\cdots\otimes q_g=c$ determines $q_0$ from the
other $g$ labels.  Every fusion multiplicity is one, so $E$ is a line.

In Fibonacci, fusing the $g$ specified pairs first leaves
$\tau^{\otimes m}$ from the channels equal to $\tau$ and the last two
anyons unfused, so
$E\cong\Hom(c,\tau^{\otimes m}\otimes\tau\otimes\tau)$.  Splitting the
last pair into $\one\oplus\tau$ gives \eqref{eq:E}, and
\eqref{eq:fibdim} with its convention at $m=0$ proves the nonvanishing
assertions.

The half-twist at $a_0$ is the scalar $r_{\one}$ or $r_\tau$ in its two
pair channels, giving \eqref{eq:relative}.  The ratio $r_{\one}/r_\tau$
has order ten, and so does $(r_{\one}/r_\tau)^3=e^{9\pi i/5}$, so the
liftable cubes realize all its powers and the projective order on $E$ is
exactly ten when both summands occur.  In Ising the last channel is
already determined and $E$ is a line, so a local braid acts there by a
scalar.  The kernel assertion is \eqref{eq:kernel}.
\end{proof}

The channel $\tau$ is noninvertible, and even after the remaining
outcome is fixed Fibonacci can retain a multiplicity space,
of dimension $\dim \cH_m^{\mathrm F}(c)$ or $\dim \cH_{m+1}^{\mathrm F}(c)$ in
\eqref{eq:E}, on each summand of which the local braid is scalar.  Its
relative phase distinguishes the two summands without acting
noncommutatively inside either.

\begin{remark}[Dimensions after fixing the liftable channels]\label{rem:count}
There are $2^g$ choices of the $g$ liftable channels.  In Ising they
give $2^g$ lines, and
$\dim \cH_{2g+2}^{\mathrm I}(c)=2^{g}$, so the lines exhaust the state space and the
geometric measurement is complete.  In Fibonacci the same $2^g$ choices
give joint eigenspaces of dimension $\dim \cH_{m+2}^{\mathrm F}(c)$, where $m$ is the
number of channels equal to $\tau$, and
\[
 \sum_{m=0}^{g}\binom{g}{m}\dim \cH_{m+2}^{\mathrm F}(c)=\dim \cH_{2g+2}^{\mathrm F}(c),
\]
that is $\sum_m\binom gm F_{m+2}=F_{2g+2}$ for $c=\tau$ and
$\sum_m\binom gm F_{m+1}=F_{2g+1}$ for $c=\one$.  The arithmetic mean
of the joint eigenspace dimensions is therefore $F_{2g+1}/2^g$ in the
vacuum sector and $F_{2g+2}/2^g$ in the $\tau$ sector.  Both increase
with $g\ge1$, whereas the Ising mean remains one.
\end{remark}

For comparison with cutting, a block decomposition into $r+s$ points has
the standard fusion-space form
\[
 \cH_{r+s}^{\mathrm F}(c)\cong\bigoplus_{a,b}
 \cH_r^{\mathrm F}(a)\otimes \cH_s^{\mathrm F}(b)\otimes\Hom(c,a\otimes b).
\]
If the braid is supported in the two blocks, it acts on each summand by
the tensor product of the block actions and the identity on the last
factor.  This supplies the quantum decomposition for an admissible cut
of Proposition~\ref{prop:cut}.  Odd block sizes do not fix the labels
$a,b$.  The total charge and the intermediate channels remain part of
the quantum data.

\section{A fixed Calabi--Yau realization}\label{sec:realization}

The preceding section concerns quantum spaces attached to a geometric
pair system.  We next check that its relative phase is
available through quasipositive braids whose liftable factorizations all
realize a single Calabi--Yau filling.

\begin{theorem}[Relative phases at fixed filling]\label{thm:fixed}
Let $k_1,\ldots,k_g$ be positive integers and $t$ a nonnegative integer,
and put
\begin{equation}\label{eq:beta}
                 \beta_t=h_1^{k_1}\cdots h_g^{k_g}h_0^{3t}.
\end{equation}
These are quasipositive braids in $L_p$.  Lifting the factors $h_i$ for
$i>0$ and $h_0^3$, and omitting the identity lifts, gives the same
positive allowable factorization for every $t$:
\[
                  \tau_{\gamma_1}^{k_1}\cdots\tau_{\gamma_g}^{k_g}.
\]
The Stein filling $W$ determined by this factorization is therefore the
same for every $t$, up to Stein deformation.  This filling is CY, with
\begin{equation}\label{eq:cost}
            b_1(W)=g,\qquad b_2(W)=\sum_{i=1}^g(k_i-1).
\end{equation}
On any fixed Fibonacci channel space $E$ of
Theorem~\ref{thm:comparison} in which both remaining channels occur,
the operators $R(\beta_t)|_E$, $0\le t\le9$, are pairwise distinct
projectively.  On each fixed Ising channel space they are projectively
the same.
\end{theorem}

\begin{proof}
All arcs are disjoint, and Lemma~\ref{lem:local} lifts each $h_i$,
$i>0$, to a positive Dehn twist and kills $h_0^3$.  The lifted curves
are nonseparating by Theorem~\ref{thm:pairs}, so the factorization is
allowable, and its ordered list of vanishing cycles is independent of
$t$, which fixes the PALF and its Stein deformation class.

In coordinates adapted to $\mathcal M$ the equations of
Theorem~\ref{thm:cy} are $s_{2i-1}=s_{2i}$ for $1\le i\le g$, which have
an integral solution whatever the multiplicities $k_i$, so $W$ is CY.
There are $M=\sum_i k_i$ vanishing cycles, of rank $g$ in homology, and
the PALF handle complex, with $2g$ one-handles and $M$ two-handles,
gives $b_1=2g-g$ and $b_2=M-g$.

On $E$, every $R(h_i)$ for $i>0$ is scalar because its pair channel has
been fixed, and the product of these scalars is independent of $t$.  The
projective variation is therefore exactly that of $R(h_0^{3t})|_E$,
computed in Theorem~\ref{thm:comparison}.
\end{proof}

\begin{example}[The first relative phase]
Take $g=1$ and the tuple $((01),(01),(01),(02))$.  Choose $a_1$ between
the first two points and $a_0$ between the last two.  For $k\ge1$,
\[
                         \beta_t=\sigma_1^k\sigma_3^{3t}
\]
always lifts to $\tau_{\gamma_1}^k$, with $b_1(W)=1$ and $b_2(W)=k-1$.
In the Fibonacci sector $c=\tau$, condition the first pair to have
channel $\tau$.  Then
\[
                       E=\cH_1^{\mathrm F}(\tau)\oplus \cH_2^{\mathrm F}(\tau)\cong\C^2.
\]
In the basis given by the two remaining-pair channels,
\[
 R(\beta_t)|_E
     =r_\tau^k\begin{pmatrix}r_{\one}^{3t}&0\\0&r_\tau^{3t}\end{pmatrix}.
\]
In either Ising total-charge sector, fixing the first pair already
gives a line, and this local operation has only a scalar effect
there.  In the Fibonacci vacuum
sector the first example with both remaining channels is $g=2$, with
both liftable pairs conditioned to channel $\tau$.
\end{example}

The word ``fixed'' matters in the Ising comparison, since on the full
Ising space $R(h_0^{3t})$ may change relative phases between different
\emph{geometric} channel outcomes, whereas the distinction proved here
holds \emph{after those outcomes have already been fixed}.

More generally, every Fibonacci braid operation has a realization by a
positive word in the standard braid generators that belongs to $L_p$.
Indeed, every standard cube $\sigma_j^3$ is liftable,
since its local Hurwitz action has order one or three.  As
$R(\sigma_j)^{10}=I$ and $3\cdot7\equiv1$, $3\cdot3\equiv-1$ modulo ten,
replacing $\sigma_j$ by $\sigma_j^{21}=(\sigma_j^3)^7$ and
$\sigma_j^{-1}$ by $\sigma_j^{9}=(\sigma_j^3)^3$ preserves the quantum
operation.  Thus $R(L_p)=R(B_n)$ for this representation.  Each cube lifts to a
positive triple Dehn twist or to the identity, and a single liftable arc
has a nonseparating lift by Theorem~\ref{thm:pairs}, so omitting the
identity factors leaves a positive allowable factorization.  These
realizations supply Stein fillings, whose CY property still requires
Theorem~\ref{thm:cy}.

For Ising, the ratio of the two pair phases has order four.  Replacing
$\sigma_j$ by $\sigma_j^9$ and $\sigma_j^{-1}$ by $\sigma_j^3$ therefore
preserves the projective operation and again uses only positive cubes.
Thus $\rho(L_p)=\rho(B_n)$ in both models.
Theorem~\ref{thm:comparison} distinguishes their actions after the
geometric channels have been fixed, and Theorem~\ref{thm:fixed}
realizes this distinction on a single CY filling.  This gives the
remaining pair a quantum role beyond the planar Ising pairing
construction.  The geometric input is the local kernel singled out by
connected binding, and the distinction between the models comes from
their imported fusion rules.

\section*{Declaration on the use of AI tools}
The author used ChatGPT (OpenAI) and Claude (Anthropic) during the
preparation of this work, to assist with the literature search, to verify
computations, for discussions that helped refine key arguments, and to
improve stylistic clarity.  All definitions, proofs, the research
direction and the final text remain entirely the author's own, and the
author assumes full responsibility for the contents.

\end{document}